\documentclass[12pt]{amsart}

\usepackage{mathrsfs}
\usepackage[colorlinks=true, pdfstartview=FitV, linkcolor=blue,
citecolor=blue,
urlcolor=blue,
breaklinks=true]{hyperref}
\usepackage{
	amsmath,
	amsfonts,
	amsthm,
	amssymb,
	paralist,
	mathrsfs,
	etoolbox,
	array,
	tikz-cd}
\usepackage[stdtext]{youngtab}

\usepackage[all]{xy}
\usetikzlibrary{arrows,patterns}

\newcommand\C{\mathbb{C}}

\newcommand\Z{\mathbb{Z}}
\newcommand\N{\mathbb{N}}

\tikzset{anchorbase/.style={>=stealth,baseline={([yshift=-0.5ex]current bounding box.center)}}}

\theoremstyle{definition}
\newtheorem{theo}{Theorem}[section]
\newtheorem{prop}[theo]{Proposition}
\newtheorem{lem}[theo]{Lemma}

\newtheorem{defin}[theo]{Definition}
\numberwithin{equation}{section}

\newtheorem{problem}{Problem}

\allowdisplaybreaks

\newtoggle{comments}
\newtoggle{details}
\newtoggle{detailsnote}

\iftoggle{comments}{%
  \newcommand{\acomments}[1]{
    \ \\
    {\color{red}
      \textbf{AS:} #1
    }
    \ \\
    }
  \newcommand{\mcomments}[1]{
    \ \\
    {\color{red}
      \textbf{MM:} #1
    }
    \ \\
    }
}{%
  \newcommand{\acomments}[1]{}
  \newcommand{\mcomments}[1]{}
}

\iftoggle{details}{%
  \newcommand{\details}[1]{
      \ \\
      {\color{green}
        \textbf{Details:} #1
      }
      \\
  }
}{%
  \newcommand{\details}[1]{}
}

\title{Higher Casimir Operators and the Okounkov Polynomials}
\author{Shivang Jindal}
\address{Department of Mathematics and Statistics, IIT Kanpur, India}
\email{jshivang@iitk.ac.in}

\begin{document}

\begin{abstract}
In this paper, we study a certain specialization of Okounkov polynomials and its connection with the eigenfunctions of higher Casimir operators \cite[14.12]{fuchs2003symmetries} for the complex semisimple lie algebra $\mathfrak{gl_n}$. 
\end{abstract}

\maketitle

\tableofcontents
\addtocontents{toc}{\protect\setcounter{tocdepth}{1}}

\section*{Acknowledgement}
The author would like to thank Professor Hadi Salmasian for his supervision and guidance through out this project. This work was done at The University of Ottawa as a part of Mitacs Globalink Research Internship.  

\section{Introduction}

We consider $\mathfrak{U(gl_n)}$, the universal envoloping algebra of the semisimple lie algebra $\mathfrak{gl_n}$. The Casimir operators are the elements of the centre of the enveloping algebra $\mathfrak{U(gl_n)}$ and are defined by $C_k := \text{tr}(\textbf{E}^k)$, where $\textbf{E} := [ E_{ij} ]$.
Thus we have \[ C_k = \sum_{i_1,i_2,\cdots,i_k} E_{i_1,i_2} \cdots E_{i_{k-1},i_k} E_{i_k,i_1}. \] It is well known  \cite{Gelfand:1950ihs} \cite[14.12]{fuchs2003symmetries} that the Casimir operators $C_1, \dots, C_n$ generate the centre of $\mathfrak{U(gl_n)}$. We set $\textbf{r} := (r_1, \dots, r_n) = (n-1, \dots, 0)$. From the theorem of highest weight, for any complex semisimple Lie algebra $\mathfrak{g}$ there is a bijection from the set of dominant integral weights to the set of equivalence classes of irreducible representation of $\mathfrak{g}$. For $\mathfrak{g} = \mathfrak{gl_n}$ this amounts to saying that there is a standard bijection between $\lambda := (\lambda_1, \dots, \lambda_n)$ satisfying $\lambda_1 \geq \lambda_2 \dots \geq \lambda_n \geq 0$ and the irreducible representations of $\mathfrak{gl_n}$.The center of $\mathfrak{U(gl_n)}$ acts as a scalar on the highest weight module $V_{\lambda}$ corresponding to $\lambda$. The eigenvalue of $C_k$ on $V_{\lambda}$ can thus be just by computing its action on the highest weight. Since we know the action of standard generators of $\mathfrak{gl_n}$ on the highest weight, using recursion we find that the eigenvalue of $C_k$ on $V_{\lambda}$ is given by \[c_k(\lambda) = \text{tr}(A^kF), \] where \[ A :=  \begin{bmatrix}\lambda_1+r_1 & -1 & -1& \cdots & -1 \\ 0 & \lambda_2+r_2 & -1 & \cdots & -1 \\ \vdots & \vdots & \ddots & \ddots & \vdots \\ 0 & \cdots & 0 & \lambda_{n-1}+r_n & -1 \\ 0 & \cdots & \cdots & 0 & \lambda_n + r_n \end{bmatrix}  \] and $F$ is the $n \times n$ matrix all of whose entries are equal to $1$ \cite[Chapter 9.4, Theorem 1]{doi:10.1142/0352}. 

On the other hand we consider \textit{Okounkov polynomials} $P_{\lambda}(x;\tau;\alpha)$ which are natural $BC_n$-type analogs of the interpolation Macdonald polynomials \cite{1996q.alg....11011O}. Let $n,r$ be such that $n \geq 2r$. Let $G_n$ denote $\text{GL}_n(\C)$ and $K_n$ denote the subgroup of unitary matrices that is \[ K_n =  \{ g \in \text{GL}_n(\C) : g^{\dagger}g = I_{n \times n} \} \]

Consider the space $Y$, the Grassmannian of $r$-dimensional subspaces of $\C^n$. $Y$ is a compact symmetric space of type $\text{BC}_r$\cite{bump2004lie}.The Group $K_n$ acts by left translation on $C^{\infty}(Y)$, the space of complex-valued smooth functions on $Y$.  Using \cite[Chapter 5]{helgason2000groups} we see that $C^{\infty}(Y)_{\textrm{K-finite}}$ decomposes as a multiplicity-free direct sum of irreducible $K_r \times K_{n-r}$ spherical $K_n$ modules and are naturally parametrized by partitions $\mu \in \mathcal{P}_r$. With the basis vectors as in \cite[12.3.2]{goodman2009symmetry} , We find that for $\mathbb{F} = \C$ , highest weight of irreducible $K_r \times K_{n-r}$-spherical $K_n$ modules is given by 
\[ \tilde{\mu} = (\mu_1,\ldots,\mu_r,\underbrace{0,\ldots,0}_{\text{$n-2r$ times}},-\mu_r,\ldots,-\mu_1) \]  \cite[12.3.2, Type CII]{goodman2009symmetry} 

Therefore as $K_n$ modules, 
\[ C^{\infty}(Y)_{\textrm{K-finite}} \simeq \bigoplus_{\mu \in \mathcal{P}_r} V_{\mu} \] 

In \cite{2016arXiv160900939S}, the authors constructs a family of $K_n$ invariant differential operators $D_{\lambda,s}$ called quadratic Capelli operators and show that specialization of Okounkov type BC interpolation polynomials are eigenvalues of these operators acting on $V_{\mu}$.  We state the result for $\mathbb{F} = \C$.

\begin{theo}[Sahi-Salmasian]\label{hadi}
For every $\lambda,\mu \in \mathcal{P}_r$ and every $s \in \C$, the operator $D_{\lambda,s}$ acts on $V_{\mu}$ by the scalar \[ c_{\lambda,s}(\mu) = y_{\lambda}P_{\lambda}\left(\mu+\rho;1 ; s- \frac{n-1}{2}\right), \] where $\rho  = (\rho_1,\rho_2,\ldots,\rho_r)$,  $\rho_i = \frac{n-(2i-1)}{2}$, and $y_{\lambda}$ is a constant and is defined as \cite[Section 4]{2016arXiv160900939S} \[ y_{\lambda} := \frac{(-4)^{|\lambda|}}{\prod_{\mathfrak{b} \in \lambda} (a_{\lambda}(\mathfrak{b})+1+l_{\lambda}(\mathfrak{b}))} \]
\end{theo}

From the above theorem, one expects that the Okounkov polynomials are related with elements of the center of enveloping algebra. In this paper, we attempt to write this specialization of Okounkov polynomials as the linear combination of the eigenvalues of the Casimir operator. Our result is the following, 

\begin{theo}\label{theorem}
Let $r,n$ be such that $2r \leq n$. Given a partition $\lambda \in \mathcal{P}_r$ where \[ \mathcal{P}_r := \{ (\lambda_1,\ldots,\lambda_r) \in \Z^{r} : \lambda_1 \geq \ldots \geq \lambda_r \geq 0 \}. \]Let $s_{\lambda} = \sum_{\mu} a_{\mu}p_{\mu}$  where $a_\mu$ are the coefficients , $s_{\lambda}$ and $p_{\lambda}$ are the Schur function and power sum symmetric function for the partition $\lambda$ in $r$ variables respectively. Then 
\[ P_{\lambda}(\mu+\rho;1;s-\frac{n-1}{2}) = \sum_{\mu} \frac{a_{\mu}}{2^{s(\mu)}} C_{2\mu} + \text{LowerOrderTerms} \]
where $C_{\lambda}$ denotes the Restricted Casimir polynomial, which we define in the Section \ref{resCas}. 
\end{theo}

The structure of the paper is as follows. 
We first give an overview of Okounkov polynomials in Section \ref{2} which are polynomials invariant under Weyl group of restricted root system \cite[Chapter 32]{bump2004lie} of type $\text{BC}_r$ and give their Combinatorial interpretation. Then we define Restricted Casimir polynomials in Section \ref{3} and finally prove the theorem in Section \ref{4}, establishing the connection between these two types of polynomials.

\section{Okounkov Polynomials} \label{2}
In this section, we define and give the combinatorial description of Okounkov polynomials. These definitions are taken directly from \cite[Section 4]{2016arXiv160900939S} and \cite{2014arXiv1408.5993K}.
\begin{defin}
Let $r \in \N$. Given a partition $\lambda = (\lambda_1,\lambda_2,\ldots,\lambda_r) \in \mathbb{Z}^r$. A Laurent polynomial $f(x_1,x_2,\ldots,x_r)$ in variables $x_1,x_2,\ldots,x_r$ is called $\lambda$-monic if the coefficient of $x_1^{\lambda_1}\ldots x_r^{\lambda_r}$ in $f(x_1,x_2,\ldots,x_r)$ is equal to $1$. 
\end{defin}

In \cite{Okounkov1998} Okounkov defined a family of Laurent polynomials \[ P_{\lambda}^{\textrm{ip}}(x;q,t,a) \in \C(q,t,a)[x_1^{\pm 1} \ldots x_r^{\pm 1}] \] parametrized by partitions $\lambda \in \mathcal{P}_r $ where every $P_{\lambda}^{ip}$ is the unique Laurent polynomial of degree $|\lambda|$ which is invariant under action of Weyl group $W = S_r \rtimes {\pm1}^r$ of type $BC_r$ on the $x_i$'s by permutation and inversions, and satisfies the vanishing condition \[ P^{\textrm{ip}}_{\lambda}(aq^{\mu}t^{\delta};q,t,a) = 0 \textrm{ unless } \lambda \subseteq \mu,\] where $\delta := (r-1,\ldots,0)$ , $\mu \in \mathcal{P}_r$ and $q^{\mu}t^{\delta} = (q^{\mu_1}t^{\delta_1}, \ldots, q^{\mu_r}t^{\delta_r})$. 

By taking limit \cite[Def. 7.1]{2014arXiv1408.5993K} $q \rightarrow 1$ of $P^{\textrm{ip}}_{\lambda}$ we have
\begin{defin}[Okounkov type BC interpolation polynomials] $P_{\lambda}(x;\tau;\alpha) \in \C(\tau,\alpha)[x_1,\ldots,x_r]$ are defined as \[ P_{\lambda}(x;\tau;\alpha) = \lim_{q \uparrow 1} (1-q)^{-2|\lambda|}P_{\lambda}^{\textrm{ip}}(q^x;q,q^{\tau},q^{\alpha}) \] where $|\lambda| := \sum_{i} \lambda_i$.\end{defin}

These polynomials are invariant under permutations and sign changes of $x_1,\ldots,x_r$. For the definitions of symmetric polynomials and their properties we refer \cite[Chapter 1]{MacSymHall}
We now give the combinatorial formula for $P_{\lambda}(x;\tau,\alpha)$ as given in \cite{2014arXiv1408.5993K} 

Every partition $\lambda$ can be represented by a Young diagram consisting of boxes $\mathfrak{b} = \mathfrak{b}(i,j)$, where \[ \frak{b}(i,j) \in \{ (p,q) \in \mathbb{Z}^2 : 1 \leq p \leq l(\lambda) \ \mbox{and} \  1 \leq q \leq \lambda_i \} .\]
Arm Length of a box $\mathfrak{b} = \mathfrak{b}(i,j)$ is $a_{\lambda}(\mathfrak{b}) = \lambda_i-j$ and leg length is $l_{\lambda}(\mathfrak{b}) = |\{ k > i : \lambda_k \geq j \} |. $
Arm co-length $a_{\lambda}'(\mathfrak{b}) = j-i$ and leg co-length $l_{\lambda}'(\mathfrak{b}) = i-1$. 

We consider the reverse tableau of shape $\lambda$ with entries in $\{ 1,2,\ldots, r\}$ which is defined as tableau \cite{Stanley:2011:ECV:2124415} with weakly decreasing rows and strongly decreasing columns. 
We now take any reverse tableau $T$ and let $\lambda^{(k)} \subset \lambda$ be the partition corresponding to the boxes $\mathfrak{b} \in \lambda$ that satisfy $T(\mathfrak{b}) > k$. 
For two partitions $\nu \subset \mu$ , denote by $(R{\backslash }C)_{\mu \backslash \nu}$ to be the set of boxes which are in a row of $\mu$ intersecting with $\mu \backslash \nu$, but not in a column of $\mu$ intersecting with $\mu \backslash \nu$. Then we have 

\[ b_{\mu}(\mathfrak{b};\tau) = \frac{a_\mu(\mathfrak{b})+\tau(l_\mu(\mathfrak{b}) + 1)}{a_\mu(\mathfrak{b}) + \tau l_\mu(\mathfrak{b}) + 1} \] and \[ \Phi_{T}(\tau) = \prod_{i=1}^{r} \prod_{\mathfrak{b} \in (R{\backslash}C)_{\lambda^{(i-1)}{\backslash}\lambda^{(i)}}}  \frac{b_{\lambda^{(i)}}(\mathfrak{b};\tau)}{b_{\lambda^{(i-1)}}(\mathfrak{b}; \tau)}. \] 

Then our polynomial $P_{\lambda}(x;\tau;\alpha)$, is  

\[ \sum_{T} \Phi_T(\tau) \prod_{\mathfrak{b} \in \lambda} (x_{T(\mathfrak{b})}^2-(a_{\lambda}'(\mathfrak{b}) + \tau(r-T(\mathfrak{b})-l_{\lambda}'(\mathfrak{b}))+\alpha)^2). \] where the sum is over all reverse tableaux $T$ of shape $\lambda$ with entries in $ \{ 1,\ldots,r \}.$

\section{Restricted Casimir Polynomials} \label{3}
\label{casimir}
In this section we state the explicit formula as given in \cite{Scheunert} for the eigenvalues of higher Casimir operators. 
Given irreducible representation $V_{\lambda}$ of $\mathfrak{gl_n}$ where $\lambda = (\lambda_1,\ldots,\lambda_n)$. Then the eigenfunction of the Casimir operator $C_p$ is defined as 
\[ c_p(\lambda) = \sum_{i} a_i\left(\lambda_i+\rho_i+\frac{n-1}{2}\right)^{p} \]
where 
\[ a_ i = \prod_{j \neq i} \left(1 - \frac{1}{\lambda_i-\lambda_j+\rho_i-\rho_j}\right) \]
and $\rho$ is half the sum of positive roots  which is $(\frac{n-1}{2},\ldots, \frac{n-(2k-1)}{2},\ldots,\frac{1-n}{2})$ in our case.

Note that for $\lambda = (\lambda_1,\ldots,\lambda_n)$ 
\[ c_p(\lambda- \rho) = \sum_{i} a_i\left(\lambda_i + \frac{n-1}{2}\right)^{p} \] where $a_i = \prod_{j \neq i} (1 - \frac{1}{\lambda_i - \lambda_j}) $
which is clearly a symmetric polynomial in variables $\lambda_1,\ldots,\lambda_n$. 

\begin{lem}
The set of polynomials $c'_i(\lambda) = c_i(\lambda - \rho)$ where $\lambda = (x_1,\ldots,x_n)$ in variables $x_1,\ldots,x_r$ for $i \geq 0$ are algebraically independent. 
\end{lem}

\begin{proof}

\[ c'_p = \sum_{i=1}^{n} \left(x_i+\frac{n-1}{2}\right)^p \prod_{j \neq i} \left(1 - \frac{1}{x_i - x_j}\right) \] 
We write it as 
\[ c'_p = \sum_{i=1}^{n} \left(x_i+\frac{n-1}{2}\right)^p\left(\prod_{j \neq i}\left(1 - \frac{1}{x_i - x_j}\right)-1\right) + \sum_{i=1} \left(x_i+\frac{n-1}{2}\right)^p .\] 
Since $ \sum_{i=1}^{n} (x_i+\frac{n-1}{2})^p$ is a symmetric polynomial,
\[ (c'_p)^{*} = \sum_{i=1}^{n} \left(x_i+\frac{n-1}{2}\right)^p \left(\prod_{j \neq i} \left(1 - \frac{1}{x_i - x_j}\right)-1\right) \] is also symmetric. We claim that degree of polynomial $(c'_p)^*$ is strictly less than $p$. Let \[ a_{i,j} = \frac{1}{x_i-x_j}\]Then, we can write \begin{align*} (c'_p)^* = \sum_{i=1}^{n} \left(x_i+\frac{n-1}{2}\right)^p\left(\prod_{j \neq i}\left(1-a_{i,j}\right) - 1\right) \\ = \sum_{i=1}^{n} \left(x_i+\frac{n-1}{2}\right)^p\left(\sum_{B\subseteq \{1,\ldots,i-1,i+1,\ldots,n \}- \emptyset} (-1)^{|B|} \prod_{j \in B} a_{i,j}\right) \end{align*} Now suppose that there is a non-zero degree $p$ term in the polynomial expression for $(c'_p)^*$. Then the polynomial $\frac{(c'_p)^{*}}{\prod_{i \neq j}a_{i,j}}$ will also have a non-zero degree $p+ \frac{n(n-1)}{2}$ term. But since we can't take $B$ to be empty set, the above expression shows that degree of each term in $\frac{(c'_p)^*}{\prod_{i \neq j}a_{i,j}}$ is strictly less than $p+\frac{n(n-1)}{2}$ which is a contradiction. So the term with highest degree $p$ in $c'_p$ is just 
\[ \sum_{i=1}^{n} (x_i)^p \]Thus $c'_p = p_p(x_1,\ldots,x_n) + \textrm{LowerDegreeTerms}$. Where $p_p(x_1,\ldots,x_n)$ denotes the $p$th degree power sum symmetric polynomial in variables $x_1,\ldots,x_n$. Since power sum symmetric functions are algebraically independent we conclude that polynomials $c'_p$ for $p \geq 0$ are algebraically independent. 
\end{proof}

The Casimir operator $C_{p}$ acts on the highest weight modules $V_{\mu}$. Theorem \ref{hadi} shows that the quadratic Casimir operators \cite{2016arXiv160900939S} acts on highest weight modules $V_\mu$ by a scalar and the eigenvalues are the specialization $P_{\lambda}(\mu+\rho;1;s-\frac{n-1}{2})$ of the Okounkov polynomials. Thus using the Harish-chandra theorem \cite[23.3]{humphreys1972introduction} we see that polynomials $P_{\lambda}(\mu+\rho;1;s-\frac{n-1}{2})$ are in linear span of the eigenvalues of Casimir operators. This amounts to restricting these eigenvalues to the highest weight of type $\tilde{\mu}$. Thus we now work on the restriction of these eigenvalues on the highest weight 
\[ \tilde{x} = (x_1,\ldots,x_r,\underbrace{0,\ldots,0}_{\text{$n-2r$ times}},-x_r,\ldots,-x_1) \ \  \textrm{where} \ \  x_1 \geq x_2 \geq \ldots x_r \geq 0 \] 

\begin{defin}[Restricted Casimir Polynomials] \label{resCas}
Given $n,r$. For any $i \in \N$ we define the Restricted Casimir polynomial to be the restriction of the eigenvalue of Casimir operator $C_i$ to the highest weight $\tilde{x}$ .
\[ C_i(x_1,\ldots,x_r) = c_i(x_1,\ldots,x_r,\underbrace{0,\ldots,0}_{n- 2r \textrm{ times}},-x_r,\ldots,-x_1) \]
Then for any partition $\lambda \in \mathcal{P}_r$ we define the Restricted Casimir polynomial for partition $\lambda$ as \[ C_{\lambda} = C_{\lambda_1}\ldots C_{\lambda_r} \]
\end{defin}

This leads to a natural question: 

\begin{problem}\label{problem}
Given $n,r,s$ with $n \geq 2r \geq 0 ,  n,r \in \N , s \in \C$ and a partition $\lambda \in \mathcal{P}_r$. Can we obtain an explicit formula for writing $P_{\lambda}(\mu+\rho;1;s-\frac{n-1}{2})$ as a linear combination of the Restricted Casimir polynomials $C_{\mu}$?
\end{problem}

We provide a partial answer to this. But before that we will show that it is enough to consider the even partitions and in that case, the linear combination will be unique. 

\begin{lem}\label{lma}
Given $i \in \N$, the polynomials $C_{2i}(x_1,\ldots,x_r)$ for all $i \geq 0$ are algebraically independent.
\end{lem}
\begin{proof}
Consider the polynomials $c_{2i}(x_1,\ldots,x_r,0,\ldots,0,-x_r,\ldots,-x_1)$. 
From the proof of previous lemma it follows that for $\lambda = (\lambda_1,\ldots,\lambda_n)$ we have 
\[ c_{2i}(\lambda_1,\ldots,\lambda_n) = p_{2i}(\lambda_1+\rho_1,\ldots,\lambda_n+\rho_n) + \textrm{LowerDegreeTerms} = F(\lambda_1+\rho_1,\ldots,\lambda_1+\rho_n) \] where $F$ is a symmetric polynomial in $n$ variables. 
We now consider highest weights $\lambda$ of the form $\tilde{x}$. We have 
\[ c_{2i}(\tilde{x}) = F(x_1+\rho_1,\ldots,x_r+\rho_r,\rho_{r+1},\ldots,\rho_{n-r},-x_r+\rho_{n-r+1},\ldots,-x_n+\rho_{1}) \]
But note that $\rho_{i} = - \rho_{n+1-i}$ and $\rho_{i} + \frac{2i-1}{2} = \frac{n}{2}$ for all $1 \leq i \leq n$. Thus, 
\[ c_{2i}(\tilde{x}+\delta) = F(x_1+\frac{n}{2},\ldots, x_r+\frac{n}{2},\rho_{r+1},\ldots,\rho_{n-r},-x_r-\frac{n}{2},\ldots,-x_1-\frac{n}{2}) \] where $\delta = (\frac{1}{2},\ldots,\frac{2r-1}{2},0,\ldots,0,-\frac{2r-1}{2},\ldots,-\frac{1}{2})$. 
Clearly this is a symmetric polynomial in variables $x_1,\ldots,x_r$.

Thus we can write 
\begin{align*}
C_{2i}\left(x_1+\frac{1}{2},\ldots,x_r+\frac{2r-1}{2}\right) &= \\ p_{2i}\biggl(x_1+\frac{n}{2},\ldots,x_r+\frac{n}{2},\rho_{r+1},\ldots,
  & \rho_{n-r},-(x_r+\frac{n}{2}),\ldots,-(x_1+\frac{n}{2})\biggr)+\textrm{LowerDegreeTerms} 
\end{align*}

So we have \[C_{2i}\left(x_1+\frac{1}{2},\ldots,x_r+\frac{2r-1}{2}\right) \in \C[a_0,a_1,\ldots,a_{2i}] \] where 
\[ a_i = p_i\left(x_1+\frac{n}{2},\ldots,x_r+\frac{n}{2},\rho_{r+1},\ldots,\rho_{n-r},-(x_r+\frac{n}{2}),\ldots,-(x_r+\frac{n}{2})\right) \]
But for odd $i$, \[ p_i\left(x_1+\frac{n}{2},\ldots,x_r+\frac{n}{2},\rho_{r+1},\ldots,\rho_{n-r},-(x_r+\frac{n}{2}),\ldots,-(x_r+\frac{n}{2})\right) = 0 ,\]  and for even $i$,  \[ p_i\left(x_1+\frac{n}{2},\ldots,x_r+\frac{n}{2},\rho_{r+1},\ldots,\rho_{n-r},-(x_r+\frac{n}{2}),\ldots,-(x_1+\frac{n}{2})\right) = 2p_{i}\left(x_1+\frac{n}{2},\ldots,x_r+\frac{n}{2}\right) .\]
Thus $C_{2i}\left(x_1+\frac{1}{2},\ldots,x_r+\frac{2r-1}{2}\right)$ is in algebra generated by polynomials $p_{2k}(x_1+\frac{n}{2},\ldots,x_r+\frac{n}{2})$ for $0 \leq k \leq i$. 
We know that power sum symmetric functions are algebraically independent and since highest degree term of $C_{2i}(x_1+\frac{1}{2},\ldots,x_r+\frac{2r-1}{2})$ is just \[ 2p_{2i}\left(x_1+\frac{n}{2},\ldots,x_r+\frac{n}{2}\right) \] it follows that the polynomials $C_{2i}(x_1+\frac{1}{2},\ldots,x_r+\frac{2r-1}{2})$ are algebraically independent. But these are just our original polynomials $C_{2i}$ with a shift, and so polynomials $C_{2i}$ are also algebraically independent. 
\end{proof}
\begin{lem}
Given $i \in \N$, the polynomials $C_{2i+1}(x_1,\dots,x_r)$ are in algebra generated by polynomials $C_{2i}(x_1,\ldots,x_r)$ for $i \geq 0$. 
\end{lem}
\begin{proof}
We consider polynomials $c_{2i+1}(x_1,\ldots,x_r,0,\ldots,0,-x_r,\ldots,-x_1)$. From the proof of previous lemma it follows that polynomials $C_{2i+1}(x_1+\frac{1}{2},\ldots,x_r+\frac{2r-1}{2})$ are in algebra generated by $p_{2k}(x_1+\frac{n}{2},\ldots,x_r+\frac{n}{2})$ for $0 \leq k \leq i$. So it is enough to show that each of $p_{2k}(x_1+\frac{n}{2},\ldots,x_r+\frac{n}{2})$ is in algebra generated by polynomials $C_{2k}(x_1+\frac{1}{2},\ldots,x_r+\frac{2r-1}{2})$. 
We use induction to prove that. Clearly this is true for $i=0,1$. We assume that this is true for $i \leq k$ then, 
\[ C_{2k+2} = 2p_{2k+2}+\textrm{LowerDegreeTerms} \]
But by the hypothesis \[ \textrm{LowerDegreeTerms} \in \C [C_0,C_2,\ldots,C_{2k}] \]so by the above equality, \[ p_{2k+2} \in \C [C_0,C_2,\ldots,C_{2k},C_{2k+2}] \] and we are done. 
\end{proof}

Thus from above lemma's it follows that, 

\begin{prop} 
Given $n,r,s$ with $n \geq 2r \geq 0 ;  n,r \in \N , s \in \C$ and a partition $\lambda \in \mathcal{P}_r$.
Then there exist unique coefficients $b_\mu$ such that \[ P_{\lambda}(\mu+\rho;1;s-\frac{n-1}{2}) = \sum b_{\mu} C_{2\mu} \] 
\end{prop}

Thus our problem can now be reformulated as asking for the explicit expression for coefficients $b_\mu$. In the next section we give the explicit formula for $b_\mu$ when $b_{\mu}$'s are the coefficient of top-order homogeneous terms which also proves the theorem. 

\section{Proof} \label{4}

\begin{prop}\label{prop}

Let $r,n$ be such that $ n \geq 2r \geq 0; n,r \in \N; s \in \C$. Let $HSP_{\lambda}$ denote the highest degree homogeneous terms in $P_{\lambda}\left(\mu+\rho;1 ; s- \frac{n-1}{2}\right)$ Then $HSP_{\lambda}$ is the Schur polynomial $s_{\lambda}$ for partition $\lambda$ evaluated at squares. That is if \[ s_{\lambda} = \sum_{\mu} a_{\mu}p_{\mu} \] then 
\[ HSP_{\lambda} = \sum_{\mu} a_{\mu}p_{2\mu} \]
\end{prop}
\begin{proof}
We use the combinatorial definition. Note that we have $\tau =1$ and so $\Psi_T(\tau) = 1 $ for all reverse tableau $T$. Thus the top homogeneous part becomes for polynomials $P_{\lambda}(x;1;\alpha)$ becomes 
\[ \sum_{T} \prod_{\mathfrak{b} \in \lambda} x^2_{T(\mathfrak{b})} \]
where $T$ are reverse tableau of shape $\lambda$ with entries in $\{ 1,2,\ldots,r \}$.But each such $T$ can be treated as a semi-standard tableau with entries $r+1-\mathfrak{b}(i,j)$. And there is one-one map namely $\mathfrak{b}(i,j) \rightarrow r+1-\mathfrak{b}(i,j)$ between set of reverse tableau of shape $\lambda$ with entries in $ \{ 1,2,\ldots,r \}$ and set of semi-standard tableau of shape $\lambda$ with entries in $\{ 1,2,\ldots,r \}$. And thus by definition the sum \[ \sum_{T} \prod_{\mathfrak{b} \in \lambda} x^2_{T(\mathfrak{b})} \] becomes the Schur polynomial evaluated at squares, that is \[ s_{\lambda}(x_1^2,\ldots,x_r^2) \] and the statement follows. 
\end{proof}

\begin{prop}
Given $i \in \N$, the top degree homogeneous term of $C_{2i}(x_1,x_2,\ldots,x_r)$ is $2p_{2i}(x_1,x_2,\ldots,x_r)$
\end{prop}
\begin{proof}
We earlier proved in the proof of Lemma \ref{lma} that 
\[ C_{2i}\left(x_1+\frac{1}{2},\ldots,x_r+\frac{2r-1}{2}\right) = 2p_{2i}\left(x_1+\frac{n}{2},\ldots,x_r+\frac{n}{2}\right)+ \textrm{LowerDegreeTerms} \]
So 
\[ C_{2i}\left(x_1,\ldots,x_r \right) = 2p_{2i}\left(x_1+\frac{n-1}{2},\ldots,x_r+\frac{n-(2r-1)}{2}\right)+\textrm{LowerDegreeTerms} \]
But by binomial theorem 
\[ p_{2i}\left(x_1+\frac{n}{2},\ldots,x_r+\frac{n}{2}\right) = \sum_{i=1}^{r} \left(x_i+\frac{n-(2i-1)}{2}\right)^{2i} = \sum_{i=1}^{r} x_i^r + \textrm{LowerDegreeTerms} \]
Thus, it follows that 
\[ C_{2i}(x_1,\ldots,x_r) = 2p_{2i}(x_1,\ldots,x_r) + \textrm{LowerDegreeTerms} \]
\end{proof}

\begin{proof}
Let $\mu$ be any partition of length $s(\mu)$. Then using previous proposition since all such $C_{2u}$ are independent, we see that 
\[ C_{(2\mu_1,\ldots,2\mu_{s(\mu)})} = 2^{s(\mu)}p_{2\mu} + \text{Lower Order Terms} \]
But we also know due to Proposition \ref{prop} that, 
\[ P_{\lambda}(\mu+\rho;1;s-\frac{n-1}{2}) = \sum_{\mu} a_{\mu} p_{2\mu} + \text{Lower Order terms} \]
Thus \[ P_{\lambda}(\mu+\rho;1;s-\frac{n-1}{2}) = \sum_{\mu} \frac{a_{\mu}}{2^{s(\mu)}} C_{2\mu} + \text{LowerOrderTerms} \]

\end{proof}

This shows that coefficient of higher degree terms are independent of $n,s$ and in-fact for $r$ sufficiently large($r \geq |\lambda|$), the coefficients are independent of $r$. Since the coefficients of higher degree have connection with the Schur functions, it suggest that general coefficients might have connection with the Schur-$q$ functions \cite{Schur1911}.

%%%%%%%%%%%%%%%%%%%%%%%%%%%%%%%%%%%
%
%

\bibliographystyle{alpha}
\bibliography{citation}

\end{document}